\documentclass[11pt,letterpaper]{article}
\usepackage[utf8]{inputenc}
\usepackage{amsmath,amsthm}
\usepackage{amsfonts}
\usepackage{amssymb}
\usepackage{authblk}
\usepackage{hyperref}

\usepackage{cite}

\newtheorem{thm}{Theorem}

\newtheorem{remark}{Remark}
\newtheorem{cor}{Corollary}
\newtheorem{ex}{Example}
\newcommand{\ackname}{Acknowledgements}

\newcommand*{\email}[1]{
    \normalsize\href{mailto:#1}{#1}
}

\newcommand{\norml}{\left| \left|}
\newcommand{\normr}{\right| \right|}

\title{Martingale decomposition of a $L^2$ space with nonlinear stochastic integrals}
\author{Clarence Simard \\ \email{simard.clarence@uqam.ca}}
\affil{D\'epartement de math\'ematiques, Universit\'e du Qu\'ebec \`a Montr\'eal.} 

\begin{document}
\maketitle



\begin{abstract}

This paper generalizes the Kunita-Watanabe decomposition of a $L^2$ space. The generalization comes from using nonlinear stochastic integrals where the integrator is a family of continuous martingales bounded in $L^2$. This result is also the solution of an optimization problem in $L^2$. First, martingales are assumed to be stochastic integrals. Then, to get the general result, it is shown that the regularity of the family of martingales respect to its spatial parameter is inherited by the integrands in the integral representation of the martingales. Finally, an example showing how the results of this paper, with the Clark-Ocone formula, can be applied to polynomial functions of Brownian integrals.
\end{abstract}
%
%

\section{Introduction}

Nonlinear stochastic integrals, sometimes called stochastic line integrals, are stochastic integrals where the integrator is a family of semimartingales. Let $\{M(x); x\in E\subset \mathbb{R} \}$ be a family of semimartingales indexed by $x$ ($x$ is often called spatial parameter). Let $\xi$ be a predictable process with values in $E$, $\int M(dt, \xi_t)$ denotes the nonlinear stochastic integral of $\xi$ respect to $M$. To get an intuitive understanding, let $\xi$ be a simple predictable process defined by $\xi_t = \xi_{t_k} \in \mathcal{F}_{t_{k-1}}$ for $t_{k-1} <t\leq t_k$ where $t_0=0<t_1<t_2<\cdots $ and $\{\mathcal{F}_t\}_{t\geq0}$ is a filtration, then, $\int_0^TM(ds, \xi_s)=\sum_{k\geq 1}\left\{ M(t_k\wedge T, \xi_{t_k}) - M(t_{k-1}\wedge T, \xi_{t_k})\right\}$. For continuous predictable processes, the integral is defined alike the standard stochastic integral, see \cite{Kunita_1990}. 

This stochastic integral has been defined for different families of semimartingales in \cite{Chitashvili_1983}.  Those results can also be found in \cite{Chitashvili/Mania_1996} and  \cite{Chitashvili/Mania_1999}. More recently, \cite{Kuhn_2012} extended the integral to a family of semimartingales which is not necessarily continuous respect to the spatial parameter. Detailed construction of nonlinear stochastic integrals can also be found under different assumptions in \cite{Kunita_1990} and \\ \cite{Carmona/Nualart_1990}.

Applications of nonlinear stochastic integrals can be found in mathematical finance, more specifically for modeling illiquid markets. In those models, a nonlinear stochastic integral is used to defined the value of the trading portfolio in which the cash flow of a transaction is a nonlinear function of the number of shares traded. One can look at \cite{Cetin/Jarrow/Protter_2004} and \cite{Bank/Baum_2004} for examples of financial applications. The generalization of the Kunita-Watanabe decomposition found in this paper could lead to an extension of the theory of quadratic hedging developed in \cite{Schweizer_1994}. 

In this paper, we look at the problem of approximating a random variable with nonlinear stochastic integrals. This problem is written as
\[
\inf_{\theta \in \mathcal{I}^M}\mathbb{E}\left[ \left( H - \int M(dt,\theta_t)\right)^2\right]
\]
where $H$ is a square-integrable random variable, $\{M(x); x\in \mathbb{R}\}$ is a family a martingales and $\mathcal{I}^M$ is a set of integrands. Without additional assumption on $\{M(x)\}$, this is usually an optimization problem over a non-convex set. Therefore, the uniqueness of the solution or even the existence is not guaranteed. The main result of this paper is to characterize, when it exists, the solution of this problem and discuss sufficient conditions for the existence and uniqueness of the solution. It will be shown that this characterization is a generalization of the Kunita-Watanabe decomposition, \cite{Kunita/Watanabe_1967}.

The remainder of this paper is organized in the following way. The next two sections establish the conditions of the probability space, introduce some notations and state the optimizing problem to be solved. Section~\ref{sec:main_res} presents the solution of the optimization problem under the simplifying assumption that the martingale family is defined by a known family of stochastic integrals. The generalization of the Kunita-Watanabe decomposition is also discussed in this section. Section~\ref{sec:general_case} provides the preliminary results required for the general version of the main theorem which is found in Section~\ref{sec:main_thm_general}. Finally, Section~\ref{sec:example} contains an example.

\section{Probability space and notations}\label{sec:space}

Let $(\Omega, \mathcal{F}, \mathbb{F}, \mathbb{P})$ be a filtered probability space where $\mathcal{F}:=\{\mathcal{F}_t;t\geq 0\}$ is a right-continuous filtration, complete respect to $\mathbb{P}$ and where $\mathbb{F} = \lim_{t\to \infty} \mathcal{F}_t$ is the smallest $\sigma-$algebra containing all $\mathcal{F}_t$. The results of this paper require that all adapted processes  are continuous, therefore, it is assumed that for any stopping times $T$ and $\{T_n\}$ such that $T_n \uparrow T$ then $\vee_n \mathcal{F}_{T_n} = \mathcal{F}_T$. 

First, define $$\mathcal{M}=\{ X:[0,\infty)\times \Omega \rightarrow \mathbb{R} ; X \text{ is a } \text{continuous martingale adapted to } \mathcal{F}\},$$ that is, for $X\in \mathcal{M}$, $t\rightarrow X_t$ is continuous a.s.,  $\mathbb{E}[|X_t|]<\infty $ for all $t \geq 0$ and $\mathbb{E}[X_t |\mathcal{F}_s] = X_s$ for all $0\leq s \leq t$. Then, let 
$$L^2(\mathbb{P}) =\left\{Y:\Omega \rightarrow \mathbb{R}; Y\in \mathbb{F}\text{ and } ||Y||_{L^2}<\infty \right\}$$ be
the space of square-integrable random variables where $||Y||_{L^2}^2 = \mathbb{E}[Y^2]$. In this paper, it is assumed that martingales are continuous and bounded in $L^2(\mathbb{P})$, that is, they belong to the set $$\mathcal{M}^2=\{X\in \mathcal{M}; \mathbb{E}[\sup_{t\geq 0}(X_t)^2]<\infty \}.$$ However, without loss of generality, the results will be stated for $$\mathcal{M}_0^2=\{X\in \mathcal{M}^2; X_0 \equiv 0  \}$$ and the space $L^2_0(\mathbb{P}) = \{X \in L^2(\mathbb{P}); \mathbb{E}[X] = 0\}$. The general case is recovered by translating processes and random variables. The following convention is also established; since there is an isometry between random variables in $L^2(\mathbb{P})$ (resp. $L^2_0(\mathbb{P})$) and martingales in $\mathcal{M}^2$ (resp. $\mathcal{M}_0^2$), the same notation will be used to define an element of $L^2(\mathbb{P})$ (resp. $L^2_0(\mathbb{P})$) and an element of $\mathcal{M}^2$ (resp. $\mathcal{M}^2_0$). For instance, for $X\in L^2(\mathbb{P})$, $X$ defines the random variable as well as the almost sure limit, $X=\lim_{t\to \infty}X_t$, of the martingale $X:=\{X_t;t\geq 0\}\in \mathcal{M}^2$. Finally, it is assumed that $(\Omega, \mathbb{F})$ is separable, therefore, there exists a countable basis for $L^2(\mathbb{P})$.

\section{Optimization problem}\label{sec:problem}

In this section, the family of martingales, $\{M(x);x\in \mathbb{R}\}$, used as an integrator is defined and the optimizing problem is stated.

Following the assumptions on the probability space $(\Omega, \mathcal{F}, \mathbb{F}, \mathbb{P})$, every martingale in $\mathcal{M}^2$ can be written as a stochastic integral. For this part of the paper, it is assumed that the integral representation of the family of martingales is known. Consequently, conditions can be directly imposed on the integrands, which simplifies the presentation of the main result.

The family of martingales is defined as 
\begin{align*}
M:   [0,\infty) \times \mathbb{R} \times \Omega &\rightarrow \mathbb{R} \\
(t,x,\omega) &\mapsto M(t,x,\omega)
\end{align*}
such that $M(x)$ is in $\mathcal{M}^2_0$ for each $x \in \mathbb{R}.$  Following the previous convention, $M(x)\in L^2_0(\mathbb{P})$ where $M(x)=\lim_{t\to \infty} M(t,x)$ a.s.

In the following, the problem that is solved is the approximation of a square-integrable random variable $H\in L^2_0(\mathbb{P})$ with the stochastic integral $\int M(ds,\theta_s)$.  In other words, the problem to solve is:
\begin{equation}\label{eq:main_prob}
\inf_{\theta \in \mathcal{I}^M}\norml H - \int M(ds,\theta_s)\normr_{L^2}
\end{equation}
where $\mathcal{I}^M$ is a set of suitable integrands which is yet to be defined. Intuitively, this problem can be seen as minimizing the distance from a point $H$ to a curve $\int M(ds,\theta_s)$ in a linear space.

 From the conditions imposed on $(\Omega, \mathcal{F},\mathbb{F},\mathbb{P})$, there exists a martingale $B\in \mathcal{M}^2_0$ and a family of continuous predictable processes $\{\mu(x);x\in \mathbb{R}\}$, 
\begin{align}
\mu: [0,\infty)\times \mathbb{R}\times \Omega &\rightarrow \mathbb{R} \notag \\
(t,x,\omega) &\mapsto \mu(t,x,\omega), \notag
\end{align} 
such that 
\begin{equation}\label{eq:mart_representation}
M(x) =\int \mu(t,x) dB_t.
\end{equation}
The representation in Equation~\eqref{eq:mart_representation} is a direct application of the Kunita-Watanabe decomposition and the existence of a countable basis for $\mathcal{M}^2$. The later is a consequence of the assumption of separability of the space $(\Omega,\mathbb{F})$. Following the integral representation of $M := \{M(x);x\in \mathbb{R}\}$, one finds a more tractable expression for the nonlinear stochastic integral respect to $M$, that is, 
\[
\int M(dt,\theta_t) = \int \mu(t,\theta_t)dB_t.
\]

\section{Main results}\label{sec:main_res}

The main result requires the function $x\mapsto \mu(t,x)$ in the representation \eqref{eq:mart_representation} to be smooth enough. While it is a simple condition to impose, it is less so to deduce the properties of $\mu(x)$ from $M(x)$. In this section, in order to provide a clearer presentation, the main result are established under the assumption that the integrand $\mu(x)$ in the integral representation \eqref{eq:mart_representation} is known.  However, this assumption is not satisfactory in the general case since the Kunita-Watanabe decomposition only gives the existence of the integrand in the integral representation. The general case requires results that are provided in Section~\ref{sec:regularity}.

 Let $\mathcal{I}^M$ be defined as $$\mathcal{I}^M = \left\{\theta; \theta \text{ is predictable and } \int M(ds,\theta_s),  \int \frac{\partial}{\partial x}M(ds,\theta_s) \in L^2(\mathbb{P})\right\}$$
and let $$\mathcal{H}^M = \left\{\int M(ds,\theta_s) ; \theta \in \mathcal{I}^M \right\}.$$ The next theorem characterizes the solution to \eqref{eq:main_prob}.

 A simple condition for the existence of a solution is that $\mathcal{H}^M$ is closed and, since it is assumed that the integral representation \eqref{eq:mart_representation} is known, one can simply assume that $\mu(t,\mathbb{R})$ is closed a.s. for all $t\geq 0$. Note that in the following, $\mathcal{C}^{m}$ is the set of $m$ times continuously differentiable functions.

\begin{thm}[Main theorem]\label{thm:main_thm}
Let $\mu:=\{\mu(t,x);t\geq 0, x\in \mathbb{R}\}$ be a family of continuous predictable process such that $x\mapsto \mu(t,x)$ is a.s. in $\mathcal{C}^1$ for all $t\geq 0$ and let $H\in L^2_0(\mathbb{P})$. Assume that for all $x\in \mathbb{R}$, $M(x):=\int\mu(t,x)dB_t \in L^2_0(\mathbb{P})$,  $\frac{\partial}{\partial x}M(x) := \int \frac{\partial}{\partial x}\mu(t,x) dB_t \in L^2(\mathbb{P})$ and that $\mu(t,\mathbb{R})$ is closed a.s. for all $t\geq 0$. Then, there exists $\theta^H \in \mathcal{I}^M$ such that 

\[
H = \int M(dt,\theta^H_t) + L^H
\]
where $L^H = H-\int M(dt,\theta^H_t)$ is orthogonal to $\int \frac{\partial}{\partial x}M(dt,\theta^H_t)$, i.e.
\begin{equation}\label{eq:orthogon}
\mathbb{E}\left[ \left( H - \int M(dt,\theta^H_t)\right)\int \frac{\partial}{\partial x} M(dt,\theta^H_t) \right] =0. 
\end{equation}

Moreover, 
\[
\norml H - \int M(dt,\theta^H_t) \normr_{L^2} = \inf_{\theta \in \mathcal{I}^M} \norml  H - \int M(dt,\theta_t) \normr_{L^2}
\]

\end{thm}

\begin{proof}
First, take a sequence $\{X^n\}\subset \mathcal{H}^M$ such that \[\norml H - X^n\normr_{L^2}  \rightarrow \inf_{X\in \mathcal{H}^M} \norml	H - X\normr_{L^2} . \] 
Thanks to the parallelogram law,  \cite{Williams_1991} Chap. 6, there exists $X^H$ such that $X^n \rightarrow X^H$ in $L^2(\mathbb{P})$. From the assumption that $\mu(t,\mathbb{R})$ is closed, hence that $\mathcal{H}^M$ is closed, there exists $\theta^n,\theta^H\in \mathcal{I}^M$ such that $\int M(dt, \theta^n_t) = X^n$ and $\int M(dt,\theta^H_t)=X^H$.

Let $F(\epsilon) = \norml H - \int M(dt,\theta^H_t+\epsilon)\normr_{L^2}^2$, then we know that \\ $\frac{d}{d\epsilon}F(\epsilon)|_{\epsilon = 0} = 0$. We now need to show that one gets the same result by differentiating inside the norm. 

Define the sequence of stopping times 
{\small
\[
\tau_n = \inf\left\{t>0 ; \sup_{\{\phi\}, |\phi_s|<1}\left|\int_0^t\frac{\partial}{\partial x}M(ds,\theta^H_s + \phi_s)- \int_0^t \frac{\partial}{\partial x}M(ds,\theta^H_s) \right|\geq n  \right\},
\] 
}
where the supremum is taken over predictable processes $\{\phi\}$ with $|\phi_s|<1$ for all $s\geq 0$.
For each $n$, 
\begin{align}\label{eq:diff_1}
&\lim_{\epsilon \to 0}\mathbb{E}\left[ \int_0^{t\wedge \tau_n} \left( \frac{\mu(s,\theta^H_s +\epsilon) - \mu(s,\theta^H_s)}{\epsilon} - \frac{\partial }{\partial x}\mu(s,\theta^H_s) \right)^2d[B]_s\right] \notag \\
& \quad = \lim_{\epsilon \to 0}\mathbb{E}\left[ \int_0^{t\wedge \tau_n} \left( \frac{\partial}{\partial x}\mu(s,\theta^H_s +\phi_s(\epsilon)) - \frac{\partial }{\partial x}\mu(s,\theta^H_s) \right)^2d[B]_s\right] \notag \notag \\
& \quad =\mathbb{E}\left[ \lim_{\epsilon \to 0} \left(\int_0^{t\wedge \tau_n}\frac{\partial}{\partial x}M(ds,\theta^H_s +\phi_s(\epsilon)) - \int_0^{t\wedge \tau_n} \frac{\partial}{\partial x}M(ds,\theta^H_s) \right)^2\right] =0, 
\end{align}
where the process $\phi(\epsilon)$ is predictable and  $\phi_s(\epsilon) \in (\theta^H_s -\epsilon, \theta^H_s+\epsilon)$.
We see that for each $n$, $\int_0^{t\wedge \tau_n}M(ds,\theta^H_s)$ is a.s. differentiable. 

Now, from the Kunita-Watanabe decomposition, there exists a predictable process $h$ and a $L^2(\mathbb{P})$-martingale $\lambda^H$ such that $H = \int h_sdB_s + \lambda^H$. Since $\lambda^H$ is strongly orthogonal to $B$, it is also strongly orthogonal to $\int M(ds, \theta_s)$ and $\int \frac{\partial }{\partial x}M(ds,\theta_s)$ for any $\theta \in \mathcal{I}^M$. Therefore, we can set $\lambda^M \equiv 0$ without loss of generality.  

Let $F_n(t,\epsilon) = \norml \int_0^{t\wedge \tau_n}\left(h_s - \mu(ds,\theta^H_s + \epsilon) \right)dB_s \normr^2_{L^2}$. Then, using Equation \eqref{eq:diff_1}, one finds that 

\begin{equation}\label{eq:dF_e=0}
\frac{\partial}{\partial \epsilon}F_n(t,\epsilon)|_{\epsilon = 0} = (-2)\mathbb{E}\left[\int_0^{t\wedge \tau_n}\left(h_s - \mu(s,\theta^H_s) \right)\frac{\partial }{\partial x}\mu(s,\theta^H_s)d[B]_s \right] = 0.
\end{equation}

Finally, from the Cauchy-Schwartz inequality and the fact that $\theta^H\in \mathcal{I}_M$,
\begin{align}
&\mathbb{E}\left[\left(H-\int M(ds,\theta^H_s)\right) \int \frac{\partial}{\partial x}M(ds,\theta^H_s)  \right] \notag \\
 & \quad \leq \norml H-\int M(ds,\theta^H_s)\normr^2 \norml \int \frac{\partial}{\partial x}M(ds,\theta^H_s) \normr^2 <\infty . \notag 
\end{align}
This inequality allows us to take the limit respect to $n$ and $t$ in \eqref{eq:dF_e=0} to show that 
\[
\mathbb{E}\left[\left(H-\int M(ds,\theta^H_s)\right) \int \frac{\partial}{\partial x}M(ds,\theta^H_s)  \right] =0,
\]
which concludes the proof.
\end{proof}

\begin{remark}
In Theorem \ref{thm:main_thm}, the condition that $\mu(t,\mathbb{R})$ is closed is only required to asses the existence of the solution in the general setting.  For an explicit problem, if the characterization in Theorem \ref{thm:main_thm} gives rise to a minimizing integrand in $\mathcal{I}^M$, then the closedness is irrelevant. 
\end{remark}

As opposed to the Kunita-Watanabe decomposition, the characterization of $\theta^H$ in Theorem \ref{thm:main_thm} does not guarantee that a process satisfying Equation \eqref{eq:orthogon} is the minimizing integrand. Since the convexity of $\mathcal{H}^M$ is not assumed, a process satisfying Equation \eqref{eq:orthogon} could give a local minimum. A simple condition for the set $\mathcal{H}^M$ to be convex is that the function $x\mapsto \mu(t,x)$ is a.s. convex for all $t\geq 0$.

In the case where the Kunita-Watanabe decomposition of $H$ is known, the next corollary shows that the characterization of $\theta^H$ becomes an almost sure characterization. 

\begin{cor}\label{cor:as_solution}
Assume the conditions of Theorem \ref{thm:main_thm} are satisfied and that $\mathcal{I}^B\subset \mathcal{I}^M$ where $\mathcal{I}^B = \left\{ \theta: \theta \text{ is predictable and uniformly bounded}\right\}$. Let $H$ has the following Kunita-Watanabe decomposition, $H = \int h_s dB_s + \lambda^H$, then 
\[
\left( h_s- \mu(s,\theta^H_s)\right)  \frac{\partial}{\partial x}\mu(s,\theta^H_s) \equiv 0
\]
for all $s \notin \Gamma$ where $\int \mathbf{1}_{\Gamma}(t) d[B]_t = 0$.
\end{cor}

\begin{proof}
It is clear that $$\mathcal{H}^\prime = \left\{ \int \theta_s \frac{\partial}{\partial x}M(ds, \theta^H_s); \theta \in \mathcal{I}^B\right\} \subset L^2(\mathbb{P}),$$ since $\theta$ is bounded and that $\int \frac{\partial}{\partial x}M(ds, \theta^H_s)$ is in $L^2(\mathbb{P})$ from assumption. By using the definition of the orthogonal projection on linear space, one has that
\begin{equation}
\begin{array}{l}
 \norml \int h_s dB_s - \left(\int M(ds,\theta^H_s) +\int \alpha_s \frac{\partial }{\partial x}M(ds,\theta^H_s)  \right) \normr^2_{L^2} \notag \\
  = \norml \int h_s dB_s - \left(\int M(ds,\theta^H_s) -\int \alpha_s \frac{\partial }{\partial x}M(ds,\theta^H_s)  \right) \normr^2_{L^2}
\end{array}
\end{equation}
for all $\alpha \in \mathcal{I}^B$. This equality is equivalent to 
\begin{equation}
\begin{array}{l}
\mathbb{E}\left[\int \left(h_s - \mu(s, \theta_s) \right)\alpha_s \frac{\partial}{\partial x}\mu(s,\theta^H_s) d[B]_s \right] \notag \\
= \mathbb{E}\left[ \int \left(h_s - \mu(s,\theta^H_s) \right)\frac{\partial }{\partial x}\mu(s,\theta^H_s) dB_s \int \alpha_s dB_s\right] = 0
\end{array}
\end{equation}
for all $\alpha \in \mathcal{I}^B$. Since the set $\left\{\int \alpha_s dB_s;\alpha \in \mathcal{I}^B\right\}$ is dense in  $$\left\{ X\in L^2(\mathbb{P}); X = \int \theta_s dB_s\right\}$$ respect to the $L^2(\mathbb{P})$ norm, hence $\int \left(h_s -\mu(s,\theta^H_s) \right)\frac{\partial }{\partial x}\mu(s,\theta^H_s)dB_s \equiv 0$. The proof is completed by taking the norm, 
\begin{equation}
\begin{array}{l}
\norml \int \left\{h_s -\mu(s,\theta^H_s) \right\}\frac{\partial }{\partial x}\mu(s,\theta^H_s)dB_s \normr^2_{L^2} \notag \\
= \mathbb{E}\left[\int \left\{(h_s -\mu(s,\theta^H_s)) \frac{\partial }{\partial x}\mu(s,\theta^H_s)\right\}^2d[B]_s \right] = 0.
\end{array}
\end{equation}
\end{proof}

Theorem \ref{thm:main_thm} and Corollary \ref{cor:as_solution} give a generalization of the Kunita-Watanabe decomposition. Indeed, the result of Theorem~\ref{thm:main_thm}, can be written as
\[
H =  \int M(dt,\theta^H_t) + L^H
\]
where $L^H = H-\int \theta^H_t dB_t$ and satisfies
\[
\mathbb{E}\left[L^H\int \alpha_t \frac{\partial }{\partial x}M(dt,\theta^H_t) \right] = 0
\]
for all bounded predictable processes $\alpha$. This last affirmation, which is part of the proof of Corollary~\ref{cor:as_solution} shows that $L^H$ is strongly orthogonal to $\int \frac{\partial }{\partial x}M(dt,\theta^H_t) $. Finally, if one puts $\mu(t,x)= xB_t$, then, applying the above results leads to  
\[
H = \int \theta^H_t dB_t + L^H 
\]
where $L^H$ is strongly orthogonal to $B$, which is the Kunita-Watanabe decomposition of $H$. 

The following example is an application of Theorem \ref{thm:main_thm} and Corollary \ref{cor:as_solution} in a case where the minimizing integrand is known explicitly. 

\begin{ex}
Let $\mathcal{F}_t = \sigma\{W^{(1)}_s,W^{(2)}_s;0\leq s \leq t\}$ where $W^{(i)}$, $i=1,2$ are two independent Brownian motions, $\mathbb{F} = \mathcal{F}_T$ for some $T>0$ and define $W_t = \rho W^{(1)}_t + \sqrt{1-\rho^2}W^{(2)}_t, \ \ \forall t\geq 0$ with $\rho \in (0,1)$. Then, define the family of martingales $\{M(x)\}\subset \mathcal{M}_0^2$ by $M(t,x) = e^{xW_t - t\frac{x^2}{2}}-1$ for $0 \leq t \leq T$.  Finally, let $H_T = \left(W^{(1)}_T\right)^2 -T= 2 \int_0^TW^{(1)}_sdW^{(1)}_s$. In the following, we find the predictable process $\theta^H$ which solves 
\begin{equation}\label{ex:prob1}
\inf_{\theta \in \mathcal{I}^M}\mathbb{E}\left[\left(H_T - \int_0^TM(ds,\theta_s) \right)^2 \right].
\end{equation}

 Using It\^{o}'s formula one finds that $M(t,x) =  \int_0^t\left\{ M(s,x)+1 \right\}xdW_s$. Therefore, $\mu(t,x)=\left(  M(t,x)+1 \right)x$ in the integral representation and one can check that $\mu(t,x)$ satisfies the conditions in Theorem \ref{thm:main_thm}.
 
 Differentiating respect to $x$ one finds that $$\frac{\partial}{\partial x} M(t,x) = \int_0^tM(s,x)(xW_s-x^2s+1)dW_s,$$ which leads to 
 
\[
\int_0^T M(ds,\theta_s) = \int_0^T \left\{ M(s,\theta_s)+1 \right\}\theta_s dW_s
\]
and 
\[ 
 \int_0^T\frac{\partial}{\partial x}M(ds,\theta_s) = \int_0^TM(s,\theta_s)(\theta_s W_s - (\theta_s)^2s+1 ) dW_s.
\]
At this point, we find the Kunita-Watanabe decomposition for $H_T$ respect to $W$. This is done by performing an orthogonal projection on the linear space $\left\{ \int_0^T\theta dW_s\right\} \subset L^2(\mathbb{P})$. One finds that $H_T = \int_0^T2 \rho W^{(1)}_s dW_s + \lambda^H_T$. With the Kunita-Watanabe decomposition of $H_T$ known, we can use Corollary \ref{cor:as_solution} to find that 

\begin{equation}\label{eq:sol_example}
\left(2\rho W^{(1)}_s - \theta^H_s M(s,\theta^H_s) \right)\left(\theta^H_sW_s-(\theta^H_s)^2s+1\right)M(s,\theta^H_s) \equiv 0. 
\end{equation}
Defining $\theta^H$ by $\theta^H_s M(s,\theta^H_s) = 2 \rho W^{(1)}_s$ for all $s$ then 
\[
\mathbb{E}\left[ \left(H_T - \int_0^TM(ds,\theta^H_s) \right)^2 \right] = \mathbb{E}[(\lambda^H_T)^2]
\]
which shows that $\theta^H$ is the minimizer. As mention before, since one can verify that the solution to Equation~\eqref{eq:sol_example} is the solution to the optimizing problem, it is irrelevant to verify if $\mathcal{H}^M$ is closed.
\end{ex}

\section{General case}\label{sec:general_case}

In the first part of the paper, it was assumed that the family of processes $\mu$ in the integral representation of the martingales $\{M(x);x\in \mathbb{R}\}$ was known.  However, in general, one only has the existence of the integrand. Therefore, the fact that the assumptions required in Theorem~\ref{thm:main_thm} are in terms of $\mu$ lessen the practicality of the result. 

In this part of the paper, conditions are stated in terms of $M:=\{M(x);x\in \mathbb{R}\}$. Indeed, using the Kunita-Watanabe decomposition, one still has the existence of $\mu$ in the integral representation, but some properties are required to get the result.  In order to generalize Theorem~\ref{thm:main_thm}, the next section is devoted to establish different relationships between the properties of $\{M(x); x\in \mathbb{R}\}$ and the properties of the integrand in the integral representation. With those properties, the main results of Section~\ref{sec:main_res} are recovered.

\subsection{Regularity conditions}\label{sec:regularity}

Recall that the family of martingales is defined by 
\begin{align*}
M:   [0,\infty) \times \mathbb{R} \times \Omega &\rightarrow \mathbb{R} \\
(t,x,\omega) &\mapsto M(t,x,\omega)
\end{align*}
such that $M(x)$ is in $\mathcal{M}^2_0$ for each $x \in \mathbb{R}$. Using the separability assumption of $(\Omega,  \mathbb{F})$, there exists $B \in \mathcal{M}_0^2$ and a family of predictable processes $\{\mu(x); x \in \mathbb{R}\}$ such that $$M(x) = \int \mu(s,x)dB_s.$$ The existence of $\mu(x)$ is a simple application of the Kunita-Watanabe decomposition for each $x$.

To establish the results of Theorem~\ref{thm:main_thm}, the function $x\mapsto \mu(t,x)$ needs to be continuously differentiable. Here, we only have the existence of $\mu$. Fortunately, our next results link the differentiability of $x\mapsto \mu(t,x)$ with the differentiability of $x\mapsto M(x)$.

\begin{thm}\label{thm:cont}
Let $\{M(x)\}$ be a family of random variables in $L^2_0(\mathbb{P})$ such that $x\mapsto M(x)$ is a.s. in $\mathcal{C}^{0}$. Then, for each $t\geq 0$,  $x\mapsto M(t,x)$ is a.s. in $\mathcal{C}^0$ and one can take $x\mapsto \mu
 (t,x)$ to be a.s. continuous for all $t\geq 0$ in the representation $M(x) = \int \mu(t,x)dB_t$.
\end{thm}
\begin{proof}
Given a compact subset $K\subset \mathbb{R}$, let 
\[
A_n = \left\{\sup_{x,y\in K}|M(x)-M(y)|\leq n \right\}
\]
and define $M^n(x) = M(x)\mathbf{1}_{A_n}$  for all $x\in \mathbb{R}$. One has that $M^n(x)\rightarrow M(x)$ a.s. and that $|M^n(x)|\leq |M(x)|$ for all $n$ and each $x\in \mathbb{R}$. By the dominated convergence theorem, one has that $$\mathbb{E}[(M^n(x)-M(x))^2]\rightarrow 0.$$

Now, let $\{\mu(x)\}$ be such that $M(x) = \int \mu(t,x)dB_t$. For $x,y\in K$, one finds that

\begin{align}
 & \mathbb{E}\left[\int (\mu(s,x) - \mu(s,y))^2d[B]s \right]  = \mathbb{E}\left[(M(x)-M(y))^2 \right] \notag \\
 & \quad \leq  \mathbb{E}[(M(x) - M^n(x))^2]+ \mathbb{E}[(M^n(x)- M^n(y))^2] \notag \\ 
    & \qquad + \mathbb{E}[(M^n(y) - M(y))^2] \notag .
\end{align}
Using the dominated convergence theorem and the fact that $M(x)$ is a.s. in $\mathcal{C}^0$ one has that $\lim_{x \to y}\mathbb{E}[(M^n(x)-M^n(y))^2]=0$ for all $n$. Finally, from the convergence of $M^n(x)$ to $M(x)$ in $L^2(\mathbb{P})$, for all $\epsilon >0$ one can find $N_\epsilon$ such that $\mathbb{E}[(M(x)- M^n(x))^2]< \epsilon$ and  $\mathbb{E}[(M(y)- M^n(y))^2]<\epsilon$ for all $n>N_\epsilon$. 

We see that for a fixed $\omega \in \Omega$, one must have that $\mu(x+h,s)\rightarrow \mu(x,s)$ except if $s\in \Gamma(\omega)\in \mathcal{B}(\mathbb{R})$ where $\int_{\Gamma(\omega)}d[B]_s(\omega)=0$. But since $\Gamma$ is the set where the martingale $B$ is constant, one can choose $x\mapsto \mu(x,s)$ to be continuous for $s\in \Gamma$. Finally, for $s \notin \Gamma$ we have that $\mu(s,x+h)\rightarrow \mu(s,x)$ a.s.
\end{proof}

In the representation $M(x) = \int \mu(s,x)dB_s$, one has that the process $\{\mu(t,x)\}_{t\geq 0}$ is unique for each $x$. However, it is not clear if the family of processes $\{\mu(x);x\in \mathbb{R}\}$ is unique, which could lead to $M$ being ill-defined. The next result shows that if $x\mapsto M(x)$ is almost surely continuous, then the representation is unique, up to an equivalence class of processes.  This class of processes is defined by processes equal up to a null set respect to the measure $\beta : \mathcal{B}([0,\infty))\times \mathbb{F}\rightarrow [0,\infty)$ given by $\beta(A) = \int \int \mathbf{1}_{A}(s,\omega) d[B]_s\mathbb{P}(d\omega)$, where $\mathcal{B}([0,\infty ))$ is the Borel sigma-algebra on $[0,\infty)$ and  $\mathbf{1}_{A}$ is the indicator function of the set $A$.

\begin{thm}
If $x\mapsto M(x)$ is a.s. in $\mathcal{C}^0$, then the representation $M(x) = \int \mu(s,x)dB_s$ is unique except on a set $A \subset \mathcal{B} ([0,\infty)) \times \mathbb{F}$ where $\mathbb{E}[\int_{A} 1d[B]_s] = 0$.
\end{thm}
\begin{proof}
Suppose that $M(x) = \int\mu(s,x)dB_s = \int \nu(s,x)dB_s$. 
\[
\mathbb{E}\left[\int (\mu(s,x)-\nu(s,x))^2d[B]_s\right] = \mathbb{E}[(M(x) - M(x))^2]=0.
\]
Therefore, $\mu(s,x) =\nu(s,x)$ except maybe on a set $\Gamma_x\in  \mathcal{B}([0,\infty))\times \mathbb{F}$ such that $\beta(\Gamma_x) = 0$.

Let $\{x_i\}_{i\in \mathcal{I}}$ be a dense subset of $\mathbb{R}$ where $\mathcal{I}$ is a countable index set. For each $i\in \mathcal{I}$, there exists $A_i \in \mathcal{B}([0,\infty))\times \mathbb{F}$ such that $\beta(A_i) = 0$ and where $\mu(s,x_i) = \nu(s,x_i)$ on $(A_i)^c$. Since $x\mapsto \mu(s,x)$ is a.s. continuous, $\mu(s,x) = \nu(s,x)$ except maybe on $\cup_{i\in \mathcal{I}} A_i$ where $\beta(\cup_{i\in \mathcal{I}} A_i) = 0$.
\end{proof}

From the previous theorem, one can say that the integrand in the representation of $M$ is almost surely unique respect to the sigma finite measure $\beta$. One gets nonetheless that $M(x) = \int \mu(s,x)dB_s$ is well defined. For a fixed $x\in \mathbb{R}$, let $\hat{\mu}(x)$ be such that $\mu(t,x,\omega)\neq \ \hat{\mu}(t,x,\omega)$ on $A\subset \mathcal{B}([0,\infty))\times \mathbb{F}$ where $\beta(A) = 0$ and $\mu(t,x,\omega) = \ \hat{\mu}(t,x,\omega)$ on $A^c$. Assume that $\mathbb{P}(\bigcup_{t\geq 0}A(t,\cdot))>0$ where $A(t,\cdot) = \{\omega \in \Omega; (t,\omega)\in A\}$ is a cross section of $A$. For each $\omega$, one must have that $\int_{A(\cdot, \omega)}d[B]_s(\omega) = 0$, which means that $A(\cdot, \omega)$ is a set of times where the martingale $B$ is constant, hence the value of $\mu(t,x,\omega)$ for $t\in A(\cdot, \omega)$ has no impact on the value $M(x)$. On the other hand, if $\int_{A(\cdot, \omega)}d[B]_s(\omega)>0$, then one must have that $\mathbb{P}(\bigcup_{t\geq 0}A(\cdot, t))=0$ and $\hat{\mu}(x)$ is actually a modification of $\mu(x)$.

To establish conditions for the differentiability of $\mu(x)$ we first define 
\[
\mathcal{C}^{m,\delta}(K) = \left\{f: \mathbb{R}\rightarrow \mathbb{R}: \left|\frac{\partial^m }{\partial x^m}f(x) - \frac{\partial^m }{\partial x^m}f(y) \right| \leq K|x-y|^\delta \right\}
\]
which is the set of $m$ times continuously differentiable functions with $\delta$-H\"older $m$-th derivative. In the following theorem, we need $M(x)$ to be in $\mathcal{C}^{1,\delta}(K)$, but we can allow $\delta$ and $K$ to depend on $\omega$ which gives a more general result than requiring $M(x)$ to be uniformly in $\mathcal{C}^{1,\delta}(K)$.

\begin{thm}\label{thm:diff}
Let $\{M(x);x\in \mathbb{R}\}$ be a family of $L^2_0(\mathbb{P})$ random variables such that $M(x)$ is a.s. in $C^{1,\delta}(K)$, where $K$ and $\delta$ are positive $\mathbb{F}-$measurable and $K$ is in $L^2(\mathbb{P})$.  Then, $\left\{\frac{d}{dx}M(x);x\in \mathbb{R} \right\}$ is a family of $L^2(\mathbb{P})$ random variables and there exists a family of predictable processes $\{\mu(x);x\in \mathbb{R}\}$ such that $x\mapsto \mu(x)$ is a.s. in $C^1$, 
\[
M(x) = \int \mu(s,x)dB_s \text{ and } \frac{\partial }{\partial x}M(x) = \int\frac{\partial}{\partial x}\mu(s,x)dB_s.
\]
\end{thm}
\begin{proof}
First, for $h\in \mathbb{R}$ we find that 
\begin{align*}
\left|\frac{M(x+h)-M(x)}{h}-\frac{d }{d x}M(x) \right| = \left|\frac{d }{d x}M(\alpha(x,h)) -\frac{d }{d x}M(x)  \right|<K|h|^\delta,
\end{align*}
where $\alpha(x,h)$ is a random variable in $(x-h,x+h)$. Moreover, if $|h|<1$ we have that $\mathbb{E}[K^2|h|^{2\delta}]<\infty $. Then,
\begin{align}
& \lim_{h\to 0}\mathbb{E}\left[\left(\frac{M(x+h)-M(x)}{h} - \frac{d}{dx}M(x) \right)^2 \right] \notag \\
& \quad = \lim_{h \to 0}\mathbb{E}\left[ \left(\frac{d }{d x}M(\alpha(x,h)) -\frac{d }{d x}M(x)\right)^2 \right] \notag \\
& \qquad \leq \lim_{h\to 0}\mathbb{E}[K^2|h|^{2\delta}] = 0.
\end{align}
We then conclude that $\frac{d}{dx}M(x)$ is in $L^2(\mathbb{P})$ since it is a closed space. 

Now, from Theorem \ref{thm:cont}, there exists a family of predictable processes $\nu$ such that $\frac{d}{dx}M(x) = \int \nu(s,x)dB_s$. Using the same argument as above, we have that 
\begin{align}
& \lim_{h\to 0}\mathbb{E}\left[\left(\frac{M(x+h)-M(x)}{h} - \frac{d}{dx}M(x) \right)^2 \right] \notag \\
& \quad = \lim_{h \to 0}\mathbb{E}\left[ \int \left(\frac{\mu(s,x+h)-\mu(s,x)}{h} -\nu(s,x)\right)^2d[B]_s \right] = 0. \notag
\end{align}
We conclude that $x\mapsto \mu(s,x)$ is almost surely continuously differentiable almost everywhere respect to $d[B]_s$. As in Theorem \ref{thm:cont}, one can take $\mu(s,x)$ to be a.s. continuously differentiable for all $s\geq 0.$
\end{proof}

\subsection{General version of the main result}\label{sec:main_thm_general}

The last theorem provides conditions on $M(x)$ so that the integrand $\mu(x)$ in the integral representation has the regularity required in Theorem~\ref{thm:main_thm}. Therefore, it is possible to recover the main result without assuming that the integral representation is known.

\begin{thm}[Main theorem, general version]\label{thm:main_generalversion}

Let $\{M(x);x\in \mathbb{R}\}$ be a family of martingales in  $\mathcal{M}^2_0$ where $x\mapsto M(x)$ is a.s. in $C^{1,\delta}(K)$ where $\delta$ and $K$ are positive and $\mathbb{F}$-measurable  and where $K$ is in $L^2(\mathbb{P})$. Let $H\in L^2_0(\mathbb{P})$ and  assume that $\mathcal{H}^M$ is closed in $L^2(\mathbb{P})$. Then, there exists $\theta^H \in \mathcal{I}^M$  such that
\[
H = \int M(ds, \theta^H_s) + L^H,
\]
where $L^H = H - \int M(ds,\theta^H_s)$ is orthogonal to $\int \frac{\partial}{\partial x}M(ds,\theta^H_s)$, i.e.
\begin{equation}
\mathbb{E}\left[\left( H - \int M(ds,\theta^H_s)\right) \int \frac{\partial }{\partial x}M(ds,\theta^H_s)\right]=0.
\end{equation}
Moreover, \[ \norml H - \int M(ds,\theta^H_s)\normr_{L^2} = \inf_{X\in \mathcal{H}^M} ||H- X||_{L^2}.\]

\end{thm}

\begin{proof}
From Theorem~\ref{thm:diff}, $M(x) = \int \mu(t,x) dB_t$ where $x\mapsto \mu(t,x)$ is a.s. in $\mathcal{C}^1$. Finally, since $\mathcal{H}^M$ is assumed to be closed, one can apply the results of Theorem~\ref{thm:main_thm}.
\end{proof}

In conclusion, the next section shows how the above results, in conjunction with the Clark-Ocone formula, \cite{Nualart_2006}, can be applied to martingales derived from polynomial functions of Brownian martingales. 

\section{Polynomial functions of Brownian martingales}\label{sec:example}

Theorem~\ref{thm:main_generalversion} extends the result of Theorem~\ref{thm:main_thm} to the more general case where the integral version of the martingales is not known. However, for modeling and application purposes, the canonical probability space of the Brownian motion is often rich enough. On this space, all continuous martingales are Brownian integrals and, in many cases, the integrand is either determined by the model or can be identified. 

In the following, we show how one can build a family of martingales from a polynomial function of Brownian integrals and get an explicit expression for the integrand in the integral representation using the Clark-Ocone formula. Then, the solution to the minimization problem \eqref{eq:main_prob} is expressed as the zero of a polynomial function in one variable.  

Let $\mathcal{P}_{n,k}$ be the set of real polynomials of degree $k$ with $n$ variables. To denote such a polynomial, let $\{\delta^{(i)}\}_{i=1}^{N_{n,k}}$ be a sequence where $\delta^{(i)} \in \{0,1,\dots, k\}^n$, $N_{n,k} = {{k+n-1}\choose{n-1}}$,  and $\sum_{j=1}^n \delta_j^{(i)} = k$ for all $i$.  Then, for $y=(y_1,\dots, y_n)$, let $y^{\delta^{(i)}} =\Pi_{j=1}^n y_j^{\delta_j^{(i)}}$ and define $p_{\{\delta\}}(y) = \sum_{i=1}^{N_{n,k}}d_{i} y^{\delta^{(i)}} \in  \mathcal{P}_{n,k}$, where $\{d_i\}_{i=1}^{N_{n,k}}$ are real coefficients.

Now, assume that $(\Omega, \mathcal{F}, \mathbb{F}, \mathbb{P})$ is the canonical probability space of the Brownian motion $B$. Let $\Theta = (\Theta_1,\dots, \Theta_n)$ be a vector of Brownian integrals, that is, $\Theta_i = \int \theta_i(t)dB_t$ for some predictable processes $\theta_i$ and assume that for each $i$, $\mathbb{E}[|\Theta_i|^m]<\infty$ for $m=2,\dots, k$. The family of martingales $\{M(x);x\in \mathbb{R}\}$ is constructed from the polynomial function 
\[
M(x) = \sum_{i=0}^k a_ix^i p^{(i)}_{\{\delta\}} (\Theta).
\]
where $ p^{(i)}_{\{\delta\}} (y) \in \mathcal{P}_{n,k-i}$ and $\{a_i\}_{i=0}^k$ are real coefficients.
Without loss of generality, it is assumed that $\mathbb{E}[M(x)] = 0$ for all $x$. Otherwise, by defining $\tilde{M}(x) = M(x) -\mathbb{E}[M(x)]$, one gets that $\tilde{M}(x)$ is centered and is still a polynomial function of the variable $x$. From the Clark-Ocone formula, one finds that 
\[
M(x) = \int \mu(t,x)dB_t
\]
where $\mu(t,x) = \sum_{i=0}^k a_i x^i\mathbb{E}[D_t  p^{(i)}_{\{\delta\}} (\Theta)|\mathcal{F}_t]$ and where $D_tp^{(i)}_{\{\delta\}}(\Theta) = \sum_{j=1}^n\frac{\partial}{\partial y_j} p^{(i)}_{\{\delta\}}(\Theta)\theta_j(t)$ is the Malliavin derivative of $p^{(i)}_{\{\delta\}}(\Theta)$ at time $t\geq 0$.  

Now, restrict $H$ to a polynomial function of $\Theta$, i.e. $H = q_{\gamma}(\Theta)$ where $q_{\gamma}(y) \in \mathcal{P}_{n,l}$. Assuming again that $\mathbb{E}[H] = 0$, on has that 
\[
H = \int h_t dB_t
\]
where $h_t = \mathbb{E}[D_tq_{\gamma}(\Theta)|\mathcal{F}_t] = \sum_{j=1}^n \theta_j(t)\mathbb{E}[\frac{\partial}{\partial y_j}q_{\gamma}(\Theta)|\mathcal{F}_t]$. Finally, using Theorem~\ref{thm:main_thm} and Corollary~\ref{cor:as_solution}, there exists $\theta^H$ such that
\[
H = \int M(dt,\theta^H_t) + L^H
\]
where $\theta^H$ satisfies
\begin{equation}
\begin{array}{rl}

& \left( h(t) - \mu(t,\theta^H_t) \right)\frac{\partial}{\partial x}\mu(t,\theta^H_t) \equiv 0 \notag \\
\Leftrightarrow & \left( \sum_{j=1}^n\theta_j(t)\mathbb{E}[\frac{\partial}{\partial y_j}q_{\gamma}(\Theta)|\mathcal{F}_t] - \sum_{i=0}^k\tilde{a}_i (\theta^H_t)^i \right)\sum_{i=1}^k i\tilde{a}_i (\theta^H_t)^{i-1} \equiv 0 \notag
\end{array}
\end{equation}
where $\tilde{a}_i = a_i \sum_{j=1}^n\theta_j(t)\mathbb{E}[\frac{\partial}{\partial y_j} p^{(i)}_{\{\delta\}}(\Theta)|\mathcal{F}_t].$

In the simpler case where $p^{(i)}_{\{\delta\}}
(y) = p_{\{\delta\}}(y)$ for all $i = 1,\dots, N_{n,k}$ then, $\theta^H$ satisfies
\[
\sum_{i=0}^ka_i (\theta_t^H)^i = \frac{  \sum_{j=1}^n\theta_j(t)\mathbb{E}[\frac{\partial}{\partial y_j}q_{\gamma}(\Theta)|\mathcal{F}_t]}{  \sum_{j=1}^n\theta_j(t)\mathbb{E}[\frac{\partial}{\partial y_j}p_{\delta}(\Theta)|\mathcal{F}_t]}
\]
or
\[
\sum_{i=1}^k i a_i (\theta^H_t)^{i-1}=0.
\]

In the above, the conditional distribution of the $n$-dimensional Gaussian vector $\Theta|_{\mathcal{F}_t}$ is $\mathcal{N}_n(\Theta_t,\Sigma_t)$ where $\Theta_t = \left(\int_{[t,\infty)}\theta_1(t)dB_t,\dots,\int_{[t,\infty)}\theta_n(t)dB_t\right)$ and $[\Sigma_t]_{ij} = \int_{[t,\infty)}\mathbb{E}[\theta_i(t)\theta_j(t)]dt$. Hence, the integrands $h_t$ and $\mu(t,x)$ are defined by the first $k$ moments of the $n$-dimensional normal distribution $\mathcal{N}_n(\Theta_t,\Sigma_t).$

\section*{\ackname}
I would like to thank my colleagues at UQAM, Jean-Fran\c{c}ois Renaud and Anne Mackay for all the discussions. I also thank Fr\'ed\'eric Godin, Cody Hyndman and the graduate students at Concordia for their time and comments on this work.

\bibliographystyle{plain}
\bibliography{martingale_decomposition_references}
\end{document}